\documentclass[10pt,a4paper,twoside]{amsart}

\usepackage[utf8]{inputenc}
\usepackage[english]{babel}

\usepackage{amsthm}
\usepackage{amstext}
\usepackage{amsmath}
\usepackage{amsfonts}
\usepackage{amssymb}
\usepackage{mathtools}

\usepackage{tikz}
\usepackage{graphics}
\usepackage{wrapfig}
\usetikzlibrary{matrix,arrows,decorations.pathmorphing}

\usepackage[a4paper]{geometry}
\usepackage{tikz-cd}
\usepackage{epigraph}

\usepackage{hyperref}
\usepackage{cite}

\usepackage{enumerate}
\usepackage{enumitem}

\author{Matteo Tamiozzo}
\title{Algebraicity of the division points of the trifolium and related topics}
\date{}

\swapnumbers
\newtheorem{teo}[subsection]{Theorem}     
\theoremstyle{plain}                    
\newtheorem{prop}[subsection]{Proposition}    
\theoremstyle{definition}               
\newtheorem{defin}[subsection]{Definition}
\newtheorem{ex}[subsection]{Example}    
\theoremstyle{remark}                   
\newtheorem{rem}[subsection]{Remark}      

\newcommand{\real}{\mathbf{R}}     
\newcommand{\comp}{\mathbf{C}}     
\newcommand{\nat}{\mathbf{N}}      
\newcommand{\rat}{\mathbf{Q}}      
\newcommand{\integ}{\mathbf{Z}}    

\newcommand{\Addr}{{\bigskip
\footnotesize
\textsc{Department of Mathematics, Imperial College London, London SW7 2AZ, UK}

\textit{Email address:} \texttt{m.tamiozzo@imperial.ac.uk}
}}

\numberwithin{equation}{subsection}

\begin{document}

\maketitle

\tableofcontents

\section{Introduction}

\subsection{} Let $S^1\subset \mathbf{R}^2$ be the circle of center 0 and radius 1 and $P=(1, 0) \in S^1$. The following two statements have been important achievements of 19th century Mathematics. Precisely, the first one was mostly known to Gauss, and its proof was completed by Wantzel in 1837; the second one was established by Lindemann in 1882:
\begin{description}
\item[Algebraicity of division points] For every integer $n \geq 1$ the points dividing $S^1$ into $n$ parts of equal length starting from $P$ are algebraic, and their coordinates generate the abelian extensions $\rat(\zeta_n)$ of $\rat$. Hence the circle can be divided into $n$ parts of equal length with ruler and compass if and only if $n$ is a product of a power of 2 and of distinct Fermat primes.
\item[Transcendence of the length] The length of $S^1$ is transcendental. Hence the circle cannot be rectified with ruler and compass.
\end{description}

\subsection{} It was later shown by Abel that Bernoulli's lemniscate, i.e. the locus of points in the plane such that the product of their distances from the two points $(-\frac{\sqrt{2}}{2}, 0), (\frac{\sqrt{2}}{2}, 0)$ equals $\frac{1}{2}$, also has algebraic division points. The proof is harder than in the case of the circle, and led Abel to develop part of the theory of elliptic functions and to discover the first instance of the phenomenon of complex multiplication. Transcendence of the length of the lemniscate also holds, and one finds a result strikingly similar to the previous one:

\begin{description}
\item[Algebraicity of division points] For every $n \geq 1$ the points dividing Bernoulli's lemniscate into $n$ parts of equal length starting from $(0, 0)$ are algebraic, and their radius generates abelian extensions of $\rat(i)$. Moreover the lemniscate can be divided into $n$ parts of equal length with ruler and compass if and only if $n$ is a product of a power of 2 and of distinct Fermat primes.
\item[Transcendence of the length] The length of Bernoulli's lemniscate is transcendental. Hence the lemniscate cannot be rectified with ruler and compass.
\end{description}

\subsection{} In this note, mostly of an elementary nature, we investigate further instances of the above phenomena. We restricted ourselves to the study of curves arising taking successive images and preimages of the circle $S^1$ via polynomials $P \in \bar{\rat}[T]$. Our first aim is to extend the previous results to the lemniscate ``with three leaves'', i.e. the locus $K$ of points such that the product of their distances from the points $Q_j=\frac{\zeta_3^j}{2^{1/3}}, j=0, 1, 2,$ equals $\frac{1}{2}$. We will prove the following result.

\begin{teo}(see Theorem \ref{teokeipdiv})\label{3leaveintro}
For every $n \geq 1$ the points dividing $K$ into $n$ parts of equal length starting from $(0, 0)$ are algebraic, and their radius generates an extension of degree at most 2 of an abelian extension of $\rat(\zeta_3)$.
\end{teo}

\subsection{} Abel's work on the lemniscate, explained in a refined version in \cite{ros81}, relies on the study of the inverse $\phi$ of the function $l$ sending a number $0<r<1$ to the length of the arc of the lemniscate joining $(0, 0)$ and the point of radius $r$ with positive coordinates (known as \emph{lemniscate sine}, in analogy with the classical sine function). Abel proved that $\phi$ can be analytically continued to a doubly periodic holomorphic function on the complex plane, whose period is a square lattice $\Lambda$. Rosen then relates the radius of a division point to the value of the Weierstrass $\wp$ function at a torsion point of the elliptic curve $\comp/\Lambda$, which has complex multiplication by $\integ[i]$. He finally uses the theory of complex multiplication to conclude. Abel's proof was more down-to-earth, constructing explicit polynomials vanishing at the radius of a division point.

An account of the arithmetic properties of division points of the lemniscate, studied both via class field theory and via the direct analysis of \emph{lemnatomic polynomials}, can also be found in \cite{cohy14}.

\subsection{} Our proof of Theorem \ref{3leaveintro} relies on the observation that, from a modern point of view, inverting the length function $l$ and explicitly studying the properties of the inverse is not necessary. We instead interpret $l$ as an Albanese map on a suitable curve $C_K$, and relate division points of $K$ with points on $C_K$ mapping to torsion points on (a quotient of) its Albanese variety. 

In retrospect, this approach is the abstract version of the calculations done by Abel: the reason why the length integral can be inverted in his situation is that the relevant curve has genus one, hence it is isomorphic to its Albanese variety. In fact, it was exactly from the attempt to invert more general integrals that the theory of Riemann surfaces and their Albanese varieties was born. An overview of the marvellous story leading from the problem of inverting abelian integrals to the introduction of Riemann surfaces and their Jacobians, then to the generalisation of the theory to arbitrary fields motivated by the Weil conjectures, and finally to Grothendieck's construction of Picard and Hilbert schemes, is given in \cite{kle04}, \cite{kle05}.

With this theory in hand, one can avoid explicit computations and make the technique work in other cases. Further examples are discussed in section \ref{sinsp}.

In order for the strategy to succeed one needs that the differential expressing the length of the curve of interest descends to a dimension one quotient of the relevant Albanese variety. We investigate the structure of Albanese varieties which arise when dividing more general curves (Erd\H{o}s lemniscates) in Proposition \ref{structalb}. One finds always abelian varieties with complex multiplication by rings of integers of cyclotomic fields, but these do not map to elliptic curves in most cases.

\subsection{} Besides looking at ``more leaves'', another natural generalisation of the lemniscate that we consider are Cassini ovals. For these curves we show the following weaker result:

\begin{teo}(see Theorem \ref{divcasov})
Let $0<a<1$ be an \emph{algebraic} number and let $C_a=\{z \in \comp: |z^2-a^2|=1\}$. For every $n \geq 1$ there are two points $P_n, P'_n \in C_a(\bar{\rat})$ such that the length of the shortest arc joining them equals $\frac{1}{4n}l(C_a)$.
\end{teo}

\subsection{} Our calculations, joint with classical and more modern transcendence results, also allow to verify that the length of several (but not all! See section \ref{lensin}) algebraic curves over $\bar{\rat}$ we encounter is transcendental. The most general statement we can prove is the following

\begin{prop}(see Proposition \ref{transgen})
Let $a \in \bar{\rat}$ and $k \geq 1$. Then the length of $C_{a, k}=\{z \in \comp: |z^k-a^k|=1\}$ is transcendental.
\end{prop}

\subsection{} We also found the following results (some of which already known) as a byproduct of our investigations:
\begin{enumerate}
\item A proof that Weil restriction along a finite Galois extension of fields preserves finite maps under mild conditions (Proposition \ref{weilresstab}), which we were unable to locate in the literature.
\item A formula for the length of the curves $C_{a, k}$ in terms of hypergeometric functions (see section \ref{lghcak}), already established by other means in \cite{bu91}. In the case $k=2$ we deduce from this formula a very classical relation between the hypergeometric function $_2F_1\left(\frac{1}{4}, \frac{3}{4}, 1, -\right)$ and elliptic integrals, which plays a role in Ramanujan's theory of elliptic functions to alternative bases in signature 4 \cite[33.9]{ber97}.
\item A construction of a family of genus 2 curves over $\bar{\rat}$ whose Jacobian is isogenous to a product of elliptic curves (see Proposition \ref{jactotspl}). Curves with this property have been looked for (at least) since Ekedahl and Serre's question whether there exist curves of arbitrarily high genus with totally split Jacobian \cite{es93}. 
\end{enumerate}

\subsection{Geometric interpretation of elliptic integrals: the work of J.-A. Serret}\label{notforgetserr} The problem of finding a geometric interpretation of elliptic integrals (as well as of other functions such as Euler's Beta function) attracted the attention of several mathematicians in the first half of the 19th century. Among others, Legendre looked for plane algebraic curves other than the lemniscate whose arc lengths could be expressed in terms of elliptic integrals. He succeeded in finding a curve of degree 6 with this property, described in \cite[p. 590]{leg26} (see also \cite{sma13}).

No further examples could be found until the problem was taken up by Joseph Alfred Serret. In a series of papers which appeared between 1842 and 1846 on the ``Journal de Math\'ematiques pures et appliqu\'ees'' \cite{ser42}, \cite{ser43}, \cite{ser43a}, \cite{ser45}, \cite{ser45a}, \cite{ser46}, Serret first expressed certain values of the Beta function as lengths of suitable curves, nowadays known as Erd\H{o}s lemniscates. He subsequently discovered the link between arc lengths of Cassini ovals and elliptic integrals. Serret then went on to find an infinite family of curves whose arc lengths can be expressed in terms of elliptic integrals; he baptised them \emph{elliptic curves} (of the first kind), and gave in the last of the above mentioned papers an elegant geometric construction of these curves.

Serret's work was highly praised at the time: in Liouville's words \cite{liou45}, ``le M\'emoire de M. Serret renferme, comme on voit, des r\'esultats utiles, remarquables''; on his recommendation, Serret's paper \cite{ser45} became part of the \emph{Recueil des Savants \'etrangers}. It seems however that Serret's work was later almost totally neglected. While several of the curves we consider where first studied by Serret, and most of the length computations we use in this note were essentially known to him, they appear to have been repeatedly forgotten and rediscovered in the following 150 years. We will do our best in the body of this note to give precise references and credit to Serret's original work.

\subsection{} Let us now introduce some preliminary definitions and notation. Both for dividing plane curves into equal parts and for defining the Albanese map it is convenient to work with curves with a marked point, hence we give the following:

\begin{defin}
A \emph{pointed algebraic curve} is couple $(X, P)$ where $X$ is an algebraic curve over $\comp$ and $P \in X(\comp)$. If $F$ is a subfield of $\comp$, we say that $(X, P)$ is a pointed algebraic curve \emph{over F} if $C$ is defined over $F$ and $P \in X(F)$.
\end{defin}

\subsection{} In this note we will be mainly interested in plane algebraic curves over $\real$. In precise, modern terms those are reduced closed subschemes of $\mathbf{A}^2_\real$ of dimension one. They need not be irreducible, nor of \emph{pure} dimension one. We will often identify $\comp$ with $\real^2$ and, for a polynomial $P\in \comp[T]$, we will see the induced map $z \mapsto P(z)$ as a map from $\real^2$ to itself. More precisely, $P$ induces a map
\begin{align*}
\real[X, Y]\rightarrow & \real[X, Y]\\
(X, Y) \mapsto & (\mathrm{Re}(P(X+iY)), \mathrm{Im}(P(X+iY)))
\end{align*} 
which yields a map $\tilde{P}: \mathbf{A}^2_\real \rightarrow \mathbf{A}^2_\real$. In scientific terms, this is the Weil restriction from $\comp$ to $\real$ of the map from $\mathbf{A}^1_\comp$ to itself induced by $P$. We will often abuse notation and denote this map simply by $P$, and the image (resp. preimage) of a subvariety $Z$ of $\mathbf{A}^2_\real$ by $P(Z)$ (resp. $P^{-1}(Z)$).

Unless stated otherwise, all the fields in this note are subfields of $\comp$. The field $\bar{\rat}$ is the algebraic closure of $\rat$ in $\comp$.

\subsection{} The figures in this note were realised with the help of C. Tamiozzo, whom the author wishes to thank.

The author was introduced to cyclotomic polynomials and the division of the circle a decade ago by A. Santoro; he later learned about division points of the lemniscate while working at a Master project under the direction of J. Nekov\'{a}\v{r}. He wishes to express his deep gratitude to both of them; this note grew out of the author's curiosity to further explore the landscape he had been introduced to.

The results in this note are (nowadays) of very modest theoretical significance. We hope nonetheless that the reader may benefit from the concrete examples that we analyse, and be motivated to learn more in depth the Mathematics we touch upon in this note - hyperelliptic curves, their periods and Jacobians, complex multiplication, hypergeometric functions and transcendence results.

\newpage

\section{Serret curves}

We will study plane algebraic curves belonging to the following family; in view of \ref{notforgetserr} we decided to name them after J.-A. Serret.

\begin{defin}
Let $S^1=\{(x, y): x^2+y^2=1\}\subset \mathbf{A}^2_\real$ be the unit circle. The family $\mathcal{S}$ of \emph{Serret curves} is the smallest family of curves such that:
\begin{enumerate}
\item $S^1 \in \mathcal{S}$.
\item If $C \in \mathcal{S}$ and $P \in \bar{\rat}[T]$ is non constant then $P^{-1}(C) \in \mathcal{S}$.
\item If $C \in \mathcal{S}$ and $P \in \bar{\rat}[T]$ is non constant then $P(C) \in \mathcal{S}$.
\end{enumerate}
\end{defin}

Let us first of all prove that all the curves in $\mathcal{S}$ are plane \emph{algebraic} curves defined over $\bar{\rat}$.

\begin{prop}
\begin{enumerate}
\item Let $P \in \comp[T]$ be a polynomial of degree at least 1. Then the induced map $\tilde{P}: \mathbf{A}^2_\real \rightarrow \mathbf{A}^2_\real$ is finite.
\item Every curve $C \in \mathcal{S}$ is an algebraic curve defined over $\bar{\rat}$.
\end{enumerate}
\end{prop}
\begin{proof}
Let $R=\mathrm{Re}(P(X+iY))$ and $I= \mathrm{Im}(P(X+iY))$. To prove $(1)$ it is enough to show that the ring extension $\real[R, I]\hookrightarrow \real[X, Y]$ is integral. Let us first show that
\begin{equation}\label{finext}
\real[R, I]\otimes_\real \comp \hookrightarrow \real[X,Y]\otimes_\real \comp
\end{equation}
is integral. Indeed, we have an isomorphism
\begin{align*}
\comp[Z, \bar{Z}]\xrightarrow{\sim} &\real[X,Y]\otimes_\real \comp\\
(Z, \bar{Z}) \mapsto & (X+iY, X-iY)
\end{align*}
whose inverse sends $(X, Y)$ to $(\frac{Z+\bar{Z}}{2}, \frac{Z-\bar{Z}}{2i})$. Under this isomorphism, $R$ (resp. $I$) is sent to $\frac{P(Z)+\bar{P}(\bar{Z})}{2}$ (resp. $\frac{P(Z)-\bar{P}(\bar{Z})}{2i}$), hence $R+iI$ (resp. $R-iI$) goes to $P(Z)$ (resp. $\bar{P}(\bar{Z})$). The extension $\comp[P(Z), \bar{P}(\bar{Z})]\hookrightarrow\comp[Z, \bar{Z}]$ is clearly integral, hence the same is true for \eqref{finext}. Now we have a commutative diagram
\begin{center}
\begin{tikzcd}
\real[X, Y] \arrow[r, hook] & \real[X, Y]\otimes_\real \comp\\
\real[R, I] \arrow[r, hook]  \arrow[u, hook] & \real[R, I]\otimes_\real \comp  \arrow[u, hook].
\end{tikzcd}
\end{center}
The bottom arrow is integral, and we proved that the right one is, hence their composite is. Since the top arrow is injective, $\real[R, I]\hookrightarrow \real[X, Y]$ is integral. This proves (1).

Since finite surjective maps are closed and preserve dimension, (1) immediately implies that every $C \in \mathcal{S}$ is an algebraic curve in $\mathbf{A}^2_\real$. To show that it is defined over $\bar{\rat}$, suppose at first that $P \in F[T]$ with $F \subset \real\cap \bar{\rat}$. Then the above discussion is the base change to $\real$ of the story obtained replacing $\real$ (resp. $\comp$) with $F$ (resp. $F(i)$). If $F \subset \bar{\rat}$ is not contained in $\real$, up to enlarging it we can suppose it contains $i$. Hence, letting $E=F \cap \real$, we have $F=E(i)$ and we can replace $\real$ (resp. $\comp$) with $E$ (resp. $F$). In both cases we see that images and preimages of curves defined over $\bar{\rat}$ are defined over $\bar{\rat}$.
\end{proof}

In fact, the scientific version of the argument in the proof of (1) above yields the following statement.

\begin{prop}\label{weilresstab}
Let $L/K$ be a finite Galois extension of fields of characteristic zero. Let $V$, $W$ be two algebraic varieties over $L$ and $f:V\rightarrow W$ a finite morphism. Assume that every finite set of points of $W$ is contained in an affine open subset (for example, take $W$ quasi-projective). Then the induced morphism
\begin{equation*}
Res_{L/K}f: Res_{L/K}V \rightarrow Res_{L/K}W
\end{equation*}
is finite.
\end{prop}
\begin{proof}
First, we claim that it is enough to show the statement for $V, W$ affine. Indeed, let $U\xrightarrow{\iota} W$ be an affine open and $U'\xrightarrow{\iota '} V$ its preimage via $f$, which is affine since $f$ is finite. Then by \cite[Proposition 7.6.2(i)]{bolura90} $Res_{L/K}U\xrightarrow{Res_{L/K}\iota} Res_{L/K}W$ and $Res_{L/K}U'\xrightarrow{Res_{L/K}\iota '} Res_{L/K}V$ are open immersions. Moreover since Weil restriction commutes with fiber products (being right adjoint) we see that $Res_{L/K}U'=(Res_{L/K}f)^{-1}(Res_{L/K}U)$. Assuming the result for affines, we know that $Res_{L/K}f: Res_{L/K}U'\rightarrow Res_{L/K}U$ is finite. Hence it is enough to verify that $Res_{L/K}W$ is covered by opens of the form $Res_{L/K}U$, with $U\subset W$ affine open. This is true under our assumption that every finite set of points of $W$ is contained in an affine open subset. Indeed in this case the restriction of scalars of $W$ is constructed by gluing the restriction of scalars of \emph{all} affine open subschemes of $W$, as explained in the proof of \cite[Theorem 7.6.4]{bolura90}.

Under the slightly more restrictive assumption that $W$ is quasi-projective, one can also argue using the following description of Weil restriction. Let $G=Gal(L/K)$. Then one has, for any quasi-projective variety $X/L$:
\begin{equation}\label{weilrestr}
Res_{L/K}X\times_K L=\prod_{\sigma \in Gal(L/K)}X^\sigma
\end{equation}
where $X^\sigma=X \times_{L, \sigma}L$ (in fact Weil defined originally restriction of scalars by showing that the above variety descends to a variety over $K$).

Using this description one sees that any point $(x_\sigma)_{\sigma \in Gal(L/K)}$ is contained in an affine open $\prod_\sigma U^\sigma$, where $U\subset X$ is an affine open containing all the points $\sigma(x_\sigma)$.

Let us now suppose $V, W$ affine. After base change to $L$, the morphism $Res_{L/K}f$ is identified with
\begin{equation*}
\prod_{\sigma \in Gal(L/K)}f^\sigma: \prod_{\sigma \in Gal(L/K)}V^\sigma \rightarrow \prod_{\sigma \in Gal(L/K)}W^\sigma
\end{equation*}
which is finite since $f$ is. Letting $Res_{L/K}V=Spec(A), \: Res_{L/K}W=Spec(B)$ and $r: B\rightarrow A$ the map inducing $Res_{L/K}f$, we have a commutative diagram

\begin{center}
\begin{tikzcd}
A \arrow[r, hook] & A\otimes_K L\\
B \arrow[r, hook]  \arrow[u, "r"] & B\otimes_K L \arrow[u, "r\otimes_K L"].
\end{tikzcd}
\end{center}

We proved that $r\otimes_K L$ is finite, hence $A\otimes_K L$ is integral over $B$, so the same holds $A$. Since $r$ is of finite type, it follows that it is finite.
\end{proof}

\begin{rem}
We warn the reader that, as it can easily be seen from the description of Weil restriction reminded after \eqref{weilrestr}, the Weil restriction functor does \emph{not} preserve open covers, even in our situation. A counterexample can be found in \cite[A.5.3]{cogapra15}. The assumption that every finite set of points is contained in an affine open is necessary to guarantee that, at least, the system formed of restrictions of \emph{all} the affine opens of the given scheme covers its restriction of scalars, and implies that the latter is represented by a scheme \cite[Theorem 7.6.4]{bolura90}.
\end{rem}

\begin{ex}
Let us give some examples of  Serret curves which will be studied in this text.
\begin{enumerate}
\item For every $P \in \bar{\rat}[T]$ the curve $P^{-1}(S^1)$ is the level set
\begin{equation*}
P^{-1}(S^1)=\{z \in \comp : |P(z)|=1\}
\end{equation*}
which is called a \emph{polynomial lemniscate}. For $P$ monic it can be described geometrically as the locus of points in the plane whose product of the distances from the roots of $P$ equals 1.
\item If $P$ is monic of degree 2 then $P^{-1}(S^1)$ is a \emph{Cassini oval}, which will be studied in section \ref{casov}. For $P(T)=T^2-1$ one obtains Bernoulli's lemniscate (scaled of a factor of $\sqrt{2}$ with respect to the Introduction).
\item For $P(T)=T^n-1$, $P^{-1}(S^1)$ is the \emph{Erd\H{o}s lemniscate} with $n$ leaves, which we examine in the next section.
\item For $P(T)=(T+1)^n$ one obtains
\begin{equation*}
P(S^1)=\{z \in \comp : |z^{1/n}-1|=1\}
\end{equation*}
which are special cases of \emph{sinusoidal spirals} (see section \ref{sinsp}). One finds all the other sinusoidal spirals pulling back the above ones via polynomials of the form $Q(T)=T^m$, and by inverting the resulting curves with respect to $S^1$ (i.e. taking their image via the map $z \mapsto \frac{1}{\bar{z}}$).
\item Let us point out that ``limits'' of curves in $\mathcal{S}$, though they do not belong to $\mathcal{S}$ in general, nor are they algebraic, are in some cases quite remarkable objects. For example, let $P_0(T)=T$, and define inductively $P_{n+1}(T)=P_n(T)^2+T$. Then the \emph{Mandelbrot lemniscates} $P_n^{-1}(S^1)$ approximate, for $n$ going to infinity, the boundary of the \emph{Mandelbrot fractal}, which is the set of complex numbers $z$ such that the sequence $\{P_n(z)\}_{n \geq 1}$ stays bounded.
\end{enumerate}
\end{ex}

\begin{rem}
For any $P \in \comp[T]$, the curve $P(S^1)$ is clearly irreducible, since it is the image of an irreducible curve. It is less obvious, but true, that the polynomial lemniscate $P^{-1}(S^1)$ is also irreducible \cite{or18}. Notice that this implies that the sinusoidal spirals described above, which are images via $P(T)=T^n$ of $(T^{m}-1)^{-1}(S^1)$, are also irreducible, hence each of them is the vanishing locus of a polynomial with $\rat$-coefficients. Finding it seems a non obvious task. Furthermore, we do not know whether one can characterise irreducible Serret curves.
\end{rem}

\section{Erd\H{o}s lemniscates}

\begin{defin}
Let $n \in \nat$. The \emph{Erd\H{o}s lemniscate with $n$ leaves} is the plane algebraic curve
\begin{equation*}
C_n=\{z \in \comp : |z^n-1|=1\}.
\end{equation*}
In other words, it is the locus of those points $P$ in the plane such that the product of the distances between $P$ and the $n$-th roots of unity equals one. 
\end{defin}

\begin{rem}
The above curves are usually called {Erd\H{o}s lemniscates} because of a conjecture of Erd\H{o}s \cite{erhepi58} predicting that $C_n$ is the longest curve among those of the form $P^{-1}(S^1)$, for $P \in \comp[T]$ monic of degree $n$ (see Proposition \ref{erdconjspec} for a very special case; the general conjecture is wide open). In fact, these curves and their lengths were first studied by Serret in \cite{ser42}, more than a century before Erd\H{o}s stated his conjecture.
\end{rem}

\begin{ex}
\end{ex}
\begin{enumerate}
\item $C_1$ is the circle of center 1 and radius 1.
\item $C_2$ is Bernoulli's lemniscate, depicted below

\begin{figure}[h!]
\begin{center}
\includegraphics[width=0.2\textwidth]{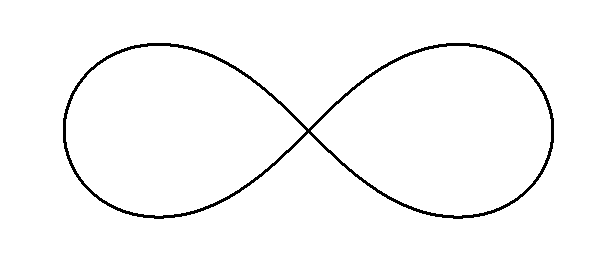}
\caption{The lemniscate of Bernoulli}
\end{center}
\end{figure}

\item $C_3$, the ``lemniscate with three leaves'' is also known as the Kiepert curve:

\begin{figure}[h!]
\begin{center}
\includegraphics[width=0.2\textwidth]{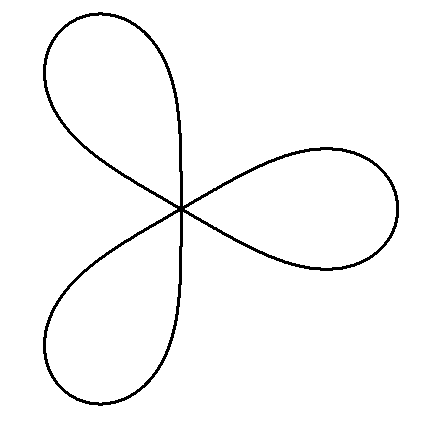}
\caption{The Kiepert curve}
\end{center}
\end{figure}
\end{enumerate}

In general, $C_n$ consists of $n$ equal leaves, cyclically permuted by the map $z \mapsto \zeta_n z$, where $\zeta_n$ is a primitive $n$-th root of unity.

\subsection{Length of an arc} Let $P_0=0$ and let $P=r(P) e^{i \theta(P)}$ be a point on $C_n$ with positive coordinates lying on the leaf crossing the axis $y=0$. Let us calculate the length $l(P)$ of the arc $\stackrel\frown{P_0P}$. The equation of $C_n$ in polar coordinates is
\begin{equation*}
r^n=2cos(n \theta).
\end{equation*}

An elementary computation yields
\begin{equation}\label{lengtherdos}
l(P)=\int\limits_0^{r(P)} \frac{1}{\sqrt{1-\left(\frac{r}{2^{1/n}}\right)^{2n}}} \mathrm{d}r= 2^{1/n} \int\limits_0^{\frac{r(P)}{2^{1/n}}} \frac{1}{\sqrt{1-s^{2n}}} \mathrm{d}s.
\end{equation}

\begin{rem} (cf. \cite{ser42})
The above formula also allows to express the total length of $C_n$ in terms of Euler's Beta function $B(a,b)=\int\limits_0^{1}t^{a-1}(1-t)^{b-1}\mathrm{d}t=\frac{\Gamma(a)\Gamma(b)}{\Gamma(a+b)}$. Indeed
\begin{align}\label{totlengtherdos}
l(C_n)&=2n\int\limits_0^{2^{1/n}} \frac{1}{\sqrt{1-\left(\frac{r}{2^{1/n}}\right)^{2n}}} \mathrm{d}r=2n2^{1/n} \int\limits_0^{1}\frac{1}{\sqrt{1-s^{2n}}} \mathrm{d}s\\
\nonumber &=2^{\frac{1}{n}}\int\limits_0^{1}\frac{1}{(1-u)^{1/2}u^{1-\frac{1}{2n}}}\mathrm{d}s=2^{\frac{1}{n}}B\left(\frac{1}{2}, \frac{1}{2n}\right).
\end{align}
\end{rem}

\subsection{} We are interested in studying the points $P_0, P_1, \ldots, P_{k-1}$ such that $l(\stackrel\frown{P_0P_j})=\frac{j}{k} l(C_n)$ for $j=1, \ldots, k-1$. More precisely, we wish to understand whether these points are algebraic. Letting $P_j=r(P_j)e^{i \theta(P_j)}$, it is enough to study whether $r(P_j)$ is algebraic. Indeed if this is the case then $P_j$ belongs to the intersection of the circle of center the origin and radius $r(P_j)$ and the curve $C_n$, which are algebraic curves defined over $\bar{\rat}$. It follows that $P_j \in C_n(\bar{\rat})$. Furthermore because of the symmetry of $C_n$ we can restrict ourselves to the case $k=2n l$, $l \in \nat$, and only consider the points $P_0, P_1, \ldots, P_{l-1}, P_l=2^{1/n}$ dividing into parts of equal length the arc $\stackrel\frown{P_0P_l}$ lying in the first quadrant. In view of equation \eqref{lengtherdos}, it is enough to study the algebraicity of the real numbers $s_0=0, s_1, \ldots, s_{l-1}$ such that
\begin{equation*}
\int\limits_0^{s_i} \frac{1}{\sqrt{1-s^{2n}}} \mathrm{d}s= i/l \int\limits_0^{1} \frac{1}{\sqrt{1-s^{2n}}}\mathrm{d}s.
\end{equation*}
The idea is to relate, in favourable situations, the numbers $s_i$ to torsion points on (quotients of) the Albanese variety of the hyperelliptic curve $y^2=1-x^{2n}$. We will study it in detail in the next section.
 
\section{Hyperelliptic curves}

\subsection{} For $n \in \nat$ let $E_n$ be the affine curve with equation $y^2=1-x^{2n}$. In the simplest case $n=1$, $E_1$ has a smooth compactification $H_1 \simeq \mathbf{P}^1$ obtained adding the two points at infinity $\infty_1, \infty_2$. The differential $\mathrm{d}x/y$ is meromorphic, with two simple poles at those points.

For $n \geq 2$, $E_n$ has a smooth compactification with two points at infinity. It can be constructed glueing $E_n$ with $E'_n:w^2=v^{2n}-1$ via the map $(x, y)\mapsto (1/x, y/x^n)$. One obtains a smooth projective curve $H_n/\rat$ of genus $n-1$. The differentials $\omega_i=x^i\mathrm{d}x/y, i=0, \ldots, n-2$, extend to holomorphic differentials on $H_n$ and give a basis of $H^0(X_n, \Omega)$.

\begin{prop}\label{Inperiod}
Let $n \geq 2$ and $0\leq i \leq n-2$. Let $I_{n, i}=\int\limits_0^{1} \frac{s^i}{\sqrt{1-s^{2n}}} \mathrm{d}s$. If $i$ is even then $4I_{n, i}$ is a period of $H_n$. In general, $I_{n, i}$ is an algebraic multiple of a period of $H_n$.
\end{prop}
\begin{proof}
Let us explain the case $i$ even, which includes the case $I_{n, 0}$ of greatest importance for us. Let
\begin{align*}
\gamma_0: [0, 1] & \rightarrow E_n\\
t & \mapsto (t, \sqrt{1-t^{2n}})\\
\gamma_1: [0, 1] & \rightarrow E_n\\
t & \mapsto (t, -\sqrt{1-t^{2n}})\\
\gamma_2: [0, 1] & \rightarrow E_n\\
t & \mapsto (-t, -\sqrt{1-t^{2n}})\\
\gamma_3: [0, 1] & \rightarrow E_n\\
t & \mapsto (-t, \sqrt{1-t^{2n}})
\end{align*}
Then $I_{n, i}=\int_{\gamma_0}\omega_i=\int_{-\gamma_1}\omega_i=\int_{\gamma_2}\omega_i=\int_{-\gamma_3}\omega_i$. Therefore
\begin{equation*}
4I_{n, i}=\int_{\gamma_0-\gamma_1+\gamma_2-\gamma_3}\omega_i
\end{equation*}
is a period, since $\gamma_0-\gamma_1+\gamma_2-\gamma_3 \in H_1(E_n, \integ)$.

For $i$ odd a similar computation works, replacing $\gamma_2$ and $\gamma_3$ with paths ending at $(\zeta, 0)$ instead of $(-1, 0)$, where $\zeta$ is a suitable $2n$-th root of unity.
\end{proof}

\begin{rem}\label{Hn1period}
\begin{enumerate}
\item For $n=1$, $i=0$, the argument in the above proof shows that $4I_{1, 0}=\int_{\gamma_0-\gamma_1+\gamma_2-\gamma_3}\omega_0$, where $\gamma_0-\gamma_1+\gamma_2-\gamma_3 \in H_1(H_1, \integ)$ and $\omega_0=\mathrm{d}x/y \in H^0(H_1, \Omega(\infty_1+\infty_2))$.
\item As in equation \eqref{totlengtherdos}, one can explicitly compute the integral in the above proposition:
\begin{equation}\label{perioderdos}
I_{n, i}=\frac{1}{2n}B\left(\frac{1}{2}, \frac{i+1}{2n}\right).
\end{equation}
\end{enumerate}
\end{rem}

\subsection{(Generalised) albanese varieties} Given a smooth projective curve $X/\comp$, its Albanese variety is the complex torus $Alb(X)=H^0(X, \Omega)^*/H_1(X, \integ)$, where $H^0(X, \Omega)^*=Hom_\comp(H^0(X, \Omega), \comp)$ and the inclusion $H_1(X, \integ)\subset H^0(X, \Omega)^*$ is given by integration of differential forms. The torus $Alb(X)$ is actually an abelian variety and, given a \emph{pointed} algebraic curve $(X, P)$, there is a canonical morphism of algebraic varieties $alb_P: X \rightarrow Alb(X)$, sending a point $Q \in X(\comp)$ to $\omega \mapsto \int\limits_P^{Q} \omega$ (well defined up to an element in $H_1(X, \integ)$). If $(X, P)$ is a pointed algebraic curve over a field $F\subset \comp$, then both $Alb(X)$ and the map $alb_P$ are defined over $F$.
\subsubsection{Functoriality} Let $f: X \rightarrow X'$ be a map of (smooth projective) algebraic curves, inducing a pullback map $f^*:H^0(X', \Omega)\rightarrow H^0(X, \Omega)$ and a pushforward map $f_*:H_1(X, \integ) \rightarrow H_1(X', \integ)$. The adjunction formula
\begin{equation*}
\int_\gamma f^*\omega= \int_{f_*\gamma}\omega
\end{equation*}
implies that the dual of $f^*$ induces a morphism $alb(f): Alb(X)\rightarrow Alb(X')$. Furthermore, if $f(P)=P'$, then $alb(f)$ fits into a commutative diagram
\begin{center}
\begin{tikzcd}
X \arrow[d, "f"] \arrow[r, "alb_P"] & Alb(X)\arrow[d, "alb(f)"] \\
X' \arrow[r, "alb_{P'}"] & Alb(X')
\end{tikzcd}
\end{center}
\begin{ex}\label{alb2}
Let $X=H_2$. Then $X$ has genus one, hence the Albanese map $alb_{(0, 1)}: X \rightarrow Alb(X)$ is an isomorphism. If $0 < t < 1$, then $alb_{(0, 1)}((t, \sqrt{1-t^4}))=\int\limits_0^{t}\frac{1}{\sqrt{1-s^{4}}} \mathrm{d}s$ (under the identification of $H^0(X, \Omega)^*$ with $\comp$ coming from choosing $\omega_0=\mathrm{d}x/y$ as a basis of $H^0(X, \Omega)$).
\end{ex}
\subsubsection{Albanese with modulus (an easy case)} To deal with the situation described in remark \ref{Hn1period} we will need a slightly more general version of the above theory. Let $X$ be a smooth projective complex algebraic curve as above, and let $D$ be an effective divisor on $X$. Let $\Omega(D)$ be the sheaf of meromorphic differentials on $X$ with poles bounded by $D$. The generalised Albanese variety of $X$ with modulus $D$ is the quotient $Alb(X, D)=H^0(X, \Omega(D))^*/H_1(X\setminus |D|, \integ)$, where $|D|$ is the support of $D$. It is an algebraic variety and the choice of a point $P \in X(\comp)\setminus |D|$ determines a map $alb_P: X\setminus |D| \rightarrow Alb(X, D)$ described as before on complex points. Everything descends to $F \subset \comp$ if $X, D$ and $P$ are defined over $F$.

There is a natural surjective map $Alb(X, D)\rightarrow Alb(X)$ whose kernel is an affine algebraic group of dimension $deg(D)-1$. If all the points in the support of $D$ have multiplicity one, then this kernel is isomorphic to $\mathbf{G}_m^{deg(D)-1}$ (see \cite[Section 17, p. 96]{ser12}).

\begin{ex}\label{alb1}
We will only need the following elementary example of generalised Albanese variety: let $D=H_1\smallsetminus E_1 \subset H_1$. Then $Alb(H_1, D)(\comp)=H^0(H_1, \Omega(D))^*/H_1(E_1, \integ) \simeq \comp/ 2\pi i \integ \xrightarrow{exp} \comp ^\times=\mathbf{G}_m(\comp)$. The Albanese map $alb_{(0, 1)}: E_1 \rightarrow \mathbf{G}_m$ is an isomorphism. If $0 < t < 1$ then $alb_{(0, 1)}((t, \sqrt{1-t^2}))=\int\limits_0^{t}\frac{1}{\sqrt{1-s^{2}}} \mathrm{d}s=arcsin(t)$.
\end{ex}

\subsection{Structure of $Alb(H_p)$, $p$ odd prime} Let $p$ be an odd prime. In this section we will determine the structure of the abelian variety $Alb(H_p)$. The key point is that the curve $E_p: y^2=1-x^{2p}$ has a large automorphism group, which forces the Albanese variety $Alb(H_p)$ to be of $CM$-type. Let us remark that $H_p$ is a quotient of the Fermat curve $x^{2p}+y^{2p}=z^{2p}$ hence, for most primes $p$, the structure of its Albanese variety can be deduced from the main result in \cite{ao91} (generalizing \cite{koro78}). The main tool we will use is the same as the one employed in \emph{loc. cit.}, i. e. a criterion for a $CM$ abelian variety to be simple in terms of its $CM$ type which goes back to Shimura-Taniyama \cite[Proposition 26, p. 69]{shta61}. However our situation is simple enough to make further technical details much easier than in the general case studied in \cite{ao91}.

Besides the hyperelliptic involution
\begin{align*}
\iota: E_p & \rightarrow E_p\\
(x, y) & \mapsto (x, -y)
\end{align*}
there is a natural action of $\integ/2p$ on $E_p$ via
\begin{align*}
[k]: E_p & \rightarrow E_p\\
(x, y) & \mapsto (\zeta_{2p}^k x, y).
\end{align*}
In particular, $E_p$ has an additional involution
\begin{align*}
\tau: E_p & \rightarrow E_p\\
(x, y) & \mapsto (- x, y).
\end{align*}
The above maps extend to automorphisms of $H_p$, and the quotient $H_p/\tau$ is isomorphic to the smooth compactification of the affine curve $y^2=1-x^p$.
\begin{prop}\label{structalb}
Let $p$ be an odd prime. Then
\begin{enumerate}
\item The quotient $H_p/\iota \tau$ is isomorphic to $H_p/\tau$.
\item $Alb(H_p)$ is isogenous to $Alb(H_p/\iota\tau)^2$.
\item The abelian variety $Alb(H_p/\iota\tau)$ is a simple abelian variety with $CM$ by $\rat(\zeta_p)$.
\end{enumerate}
\end{prop}
\begin{proof}
The automorphism group of $H_p$ contains the group $V_{2p}=\langle a, b | a^4, b^{2p}, (ab)^2, (a^{-1}b)^2 \rangle$ (in fact this inclusion is an equality according to \cite[Table 1]{mupi17}, but we won't need this), where
\begin{align*}
a: (x,y) \mapsto & \left(\frac{1}{x}, \frac{iy}{x^p}\right) \\
b: (x, y) \mapsto & (\zeta_{2p}x, y).
\end{align*}
We have $\iota=a^2$ and $\tau=b^p$. We claim that $b^p$ and $a^2b^p$ are conjugate in $V_{2p}$. Indeed:
\begin{align*}
abab=1 & \xLeftrightarrow{b^{2p}=1} a^2ba=ab^{2p-1} \Leftrightarrow a^2b=ab^{2p-1}a^{-1}\\
 & \Leftrightarrow a^2b^p=ab^{2p-1}a^{-1}b^{p-1}.
\end{align*}
Now, since $ab=b^{-1}a^{-1}$ and $a^{-1}b=b^{-1}a$, we have
\begin{equation*}
ab^{2p-1}a^{-1}b^{p-1}=ab^{2p-1}b^{-1}ab^{p-2}=ab^{2p-2}b^{-1}a^{-1}b^{p-3}=ab^{2p-3}a^{-1}b^{p-3}.
\end{equation*}
Repeating the above computation we get
\begin{equation*}
ab^{2p-1}a^{-1}b^{p-1}=ab^{2p-3}a^{-1}b^{p-3}=ab^{2p-5}a^{-1}b^{p-5}= \ldots = ab^{2p-p}a^{-1}b^{p-p}=ab^pa^{-1}.
\end{equation*}
It follows that
\begin{equation*}
a^2b^p=ab^pa^{-1}
\end{equation*}
Hence the action of $a$ induces an isomorphism $H_{2p}/\tau \sim H_{2p}/\iota \tau$. This proves $(1)$.

To prove $(2)$ we have to show that the map $H_{2p}\xrightarrow{(q_1, q_2)}H_{2p}/\tau\times H_{2p}/\iota\tau$, where $q_1, q_2$ are the two quotient maps, induces an isogeny $Alb(H_{2p})\rightarrow Alb(H_{2p}/\tau)\times Alb(H_{2p}/\iota \tau)$. For this, is suffices to show that the pullback map on differentials
\begin{equation*}
q_1^*\oplus q_2^*: H^0(H_{2p}/\tau, \Omega)\oplus H^0(H_{2p}/\iota \tau, \Omega)\rightarrow H^0(H_{2p}, \Omega)
\end{equation*}
is an isomorphism. The quotient map $q_1: H_{2p}\rightarrow H_{2p}/\tau$ (resp. $q_2: H_{2p}\rightarrow H_{2p}/\iota \tau$) induces via pullback an isomorphism $H^0(H_{2p}/\tau, \Omega)\xrightarrow{\sim}H^0(H_{2p}, \Omega)^{\tau}$ (resp. $H^0(H_{2p}/\iota \tau, \Omega)\xrightarrow{\sim}H^0(H_{2p}, \Omega)^{\iota\tau}$). Since $\tau(x, y)= (-x, y)$ (resp. $\iota \tau(x, y)= (-x, -y)$), a basis of $\tau$-invariant (resp. $\iota \tau$-invariant) differentials is given by $x\mathrm{d}x/y, x^3\mathrm{d}x/y, \ldots, x^{p-2}dx/y$ (resp. $\mathrm{d}x/y, x^2\mathrm{d}x/y, \ldots, x^{p-3}dx/y$). It follows that $q_1^*\oplus q_2^*$ is in isomorphism, and $(2)$ is proved.

To prove $(3)$ it suffices to show that $Alb(H_{2p}/\tau)$ is a simple abelian variety with $CM$ by $\rat(\zeta_p)$. Since $H_{2p}/\tau$ is the hyperelliptic curve with affine equation $u^2=1-v^p$ there is a natural action of $\integ/p$ on $H_{2p}/\tau$, inducing a map $\integ[\zeta_p]\rightarrow End(Alb(H_{2p}/\tau))$. The action of an element $[k] \in \integ/p$ on $H^0(H_{2p}/\tau, \Omega)=\comp \mathrm{d}u/v \oplus \ldots \oplus \comp u^{(p-3)/2}\mathrm{d}u/v$ is given by the diagonal matrix with entries $(\zeta_p^k, \zeta_p^{2k}, \ldots, \zeta_p^{k(p-1)/2})$. Since $[\rat(\zeta_p):\rat]=p-1=2\: \mathrm{dim} (Alb(H_{2p}/\tau))$ it follows that $Alb(H_{2p}/\tau)$ is a $CM$ abelian variety, with $CM$-type $S=\{[1], [2], \ldots, [(p-1)/2]\}\subset (\integ/p)^\times=Gal(\rat(\zeta_p)/\rat)$ (to be precise, the choice of a primitive $p$-th root of unity $\zeta_p$ induces a canonical bijection between complex embeddings of $\rat(\zeta_p)$ and $Gal(\rat(\zeta_p)/\rat)$). By \cite[Proposition 26, p. 69]{shta61}, $Alb(H_{2p}/\tau)$ is simple if and only if $\{[m] \in (\integ/p)^\times: [m]S\subset S\}=\{1\}$. Suppose there is $[m]\neq [1]$ such that $[m]S\subset S$. Then there exists $l$ such that $m2^l \leq (p-1)/2$ and $p>m2^{l+1} > (p-1)/2$, hence $[2^{l+1}] \in S$ but $[m2^{l+1}] \not \in S$. Therefore $\{[m] \in (\integ/p)^\times: [m]S\subset S\}=\{1\}$, and $(3)$ is proved.
\end{proof}

\subsection{Interlude: transcendence of lengths and periods} For $0 \leq i \leq n-2$ let us set $I_{n, i}=\int\limits_0^{1} \frac{s^i}{\sqrt{1-s^{2n}}} \mathrm{d}s=\frac{1}{2n}B\left(\frac{1}{2}, \frac{i+1}{2n}\right)$. It follows from Proposition \ref{Hn1period} that $B\left(\frac{1}{2}, \frac{i+1}{2n}\right)$ is an algebraic multiple of a period of the algebraic curve $H_n$. The structure of the Albanese $Alb(H_n)$ for $n=p$ odd prime together with consequences of W\"{u}stholz analytic subgroup theorem have the following consequences on the transcendence of $B\left(\frac{1}{2}, \frac{i+1}{2n}\right)$. Both statements in the next proposition have already been established in more generality. We report them here for the sake of completeness, and to illustrate how the general methods work in our concrete example.  We denote by $\langle x_1, \ldots, x_r\rangle_{\bar{\rat}}$ the $\bar{\rat}$-vector space generated by the complex numbers $x_1, \ldots, x_r$.

\begin{prop}\label{transcerdos}
\begin{enumerate}
\item $I_{n, i}$ is transcendental for $0\leq i \leq n-2$. In particular the length $l(C_n)$ is transcendental.
\item Let $n=p$ be an odd prime. Then $\langle I_{p, i}, 0 \leq i \leq p-2 \rangle_{\bar{\rat}}=\langle I_{p, 2j}, 0 \leq j \leq (p-3)/2 \rangle_{\bar{\rat}}$, and the dimension of this $\bar{\rat}$-vector space is $(p-1)/2$.
\end{enumerate}
\end{prop}
\begin{proof}
\begin{enumerate}
\item This follows from Schneider's transcendence result \cite{sch41} stating that if $a, b \in \rat$ and $a, b, a+b$ are not integers then $B(a, b)$ is transcendental.
\item We proved in Proposition \ref{structalb} that $Alb(H_p)$ is isogenous to $Alb(H_p/\iota\tau)^2$. It follows that the $\bar{\rat}$-vector space generated by the periods of ($\bar{\rat}$-differentials on) $Alb(H_p)$ coincides with the $\bar{\rat}$-vector space generated by the periods of $Alb(H_p/\iota \tau)$. Let $\gamma \in H_1(H_p/\iota \tau, \integ)$ be any non zero element. Then $\integ[\zeta_p]\cdot \gamma \subset H_1(H_p/\iota \tau, \integ)$ is a $\integ$-lattice of rank $p-1$, hence it has finite index in $H_1(H_p/\iota \tau, \integ)$. It follows that the $\bar{\rat}$-vector space generated by the periods of $Alb(H_p/\iota \tau)$ coincides with the vector space $\langle \int_\gamma \omega, \omega \in H^0(H_p/\iota \tau, \Omega_{\bar{\rat}}) \rangle_{\bar{\rat}}$ for any non-zero $\gamma \in H_1(H_p/\iota \tau, \integ)$. In particular we can take $\gamma=\gamma_0-\gamma_1+\gamma_2-\gamma_3$ as in the proof of Proposition \ref{Inperiod}. The claim now follows from the second point of the following spectacular result of W\"{u}stholz, which is a consequence of the analytic subgroup theorem.
\end{enumerate}
\end{proof}

\begin{prop}(cf. \cite[Proposition 1, 2]{shiwo95})\label{transperab}
\begin{enumerate}
\item Let $A$ and $B$ be abelian varieties defined over $\bar{\rat}$ and $V_A$ (resp. $V_B$) the $\bar{\rat}$-vector space generated by all periods of differentials in $H^0(A, \Omega_{\bar{\rat}})$ (resp. $H^0(B, \Omega_{\bar{\rat}})$). Then $V_A\cap V_B \neq \{0\}$ if and only if there are simple abelian subvarieties of $A$ and $B$ which are isogenous.
\item Let $A$ be a simple abelian variety over $\bar{\rat}$. Then
\begin{equation*}
dim_{\bar{\rat}}V_A=\frac{2\: dim(A)^2}{dim_{\rat}(End(A)\otimes_\integ\rat)}.
\end{equation*}
\end{enumerate}
\end{prop}

\begin{ex}
For $p=5$ we find that $B(1/2, 1/10)$ and $B(1/2, 3/10)$ are linearly independent over $\bar{\rat}$, hence their quotient $\frac{\Gamma(1/10)\Gamma(4/5)}{\Gamma(3/10)\Gamma(3/5)}$ is transcendental. On the other hand we have 
\begin{equation*}
B(1/2, 1/10)/B(1/2, 2/5)=\frac{\Gamma(\frac{1}{10})\Gamma(\frac{9}{10})}{\Gamma(\frac{2}{5})\Gamma(\frac{3}{5})}=\frac{\sin(2\pi/5)}{\sin(\pi/10)}=\sqrt{5+2\sqrt{5}}
\end{equation*}
(using that $\Gamma(z)\Gamma(1-z)=\frac{\pi}{\sin(\pi z)}$ for $z \not \in \integ$) and similarly $B(1/2, 1/5)/B(1/2, 3/10)=\sqrt{1+\frac{2}{\sqrt{5}}}$, hence $B(1/2, 1/5)$ and $B(1/2, 2/5)$ are in the $\bar{\rat}$-span of $B(1/2, 1/10)$ and $B(1/2, 3/10)$.

In general, using a similar idea Wolfart and W\"{u}stholz have shown that the only linear relations with $\bar{\rat}$-coefficients between values of the Beta function at rational parameters are given by the Deligne-Koblitz-Ogus relations \cite{wowu85}.
\end{ex}

\section{Algebraicity of division points of \texorpdfstring{$C_n, 1\leq n \leq 3$}{the trifolium}}
In this section we apply the previous results to show that the division points of $C_1, C_2$ and $C_3$ are algebraic. Of course the first two facts can be proved by more elementary means, and were already known to Gauss and Abel; however our strategy works uniformly in all the three cases.
\subsection{Gauss: division of the circle}
As discussed above, to prove that the division points of $(C_1, 0)$ are algebraic it is enough to show that, for $l \geq 1$, the real numbers $s_0=0, s_1, \ldots, s_{l-1}$ such that
\begin{equation*}
\int\limits_0^{s_i} \frac{1}{\sqrt{1-s^{2}}} \mathrm{d}s= i/l \int\limits_0^{1} \frac{1}{\sqrt{1-s^{2}}}\mathrm{d}s
\end{equation*}
are algebraic. By example \ref{alb1} and remark \ref{Hn1period}, the image of the point $(s_j, \sqrt{1-s_j^2}) \in E_1(\comp)$ via the Albanese map is a torsion point in $\mathbf{G}_m$, hence it is algebraic. It follows that the same holds for $\left(s_j, \sqrt{1-s_j^2}\right)$, hence for $s_j$. Better, since $4l$-torsion points of $\mathbf{G}_m$ are defined over $\rat(\zeta_{4l})$, one deduces that $s_j \in \rat(\zeta_{4l})$.

\begin{rem}
It may appear confusing that we are dividing the circle into $2l$ parts and the field $\rat(\zeta_{4l})$ shows up. The point is that, for example, $s_1$ is the radius of the first division point of $C_1$. For $l=2$ we get $s_1=\sqrt{2}/2$ which indeed belongs to $\rat(\zeta_8)$ and not to $\rat(\zeta_4)$. On the other hand $P_1$ has coordinates $(1,1)$, i.e. it is the complex point $1+i \in \rat(\zeta_4)$. Hence this argument alone does \emph{not} give the optimal field of definition of the points $P_j \in C_1(\bar{\rat})$. However, since $P \in C_1(\comp)$ is constructible with ruler and compass if and only if $r(P)$ is constructible, and $Gal(\rat(\zeta_{4l})/\rat)$ is a $2$-group if and only if $Gal(\rat(\zeta_{l})/\rat)$ is, one recovers Gauss' theorem on the divisibility of the circle with ruler and compass.
\end{rem}

\subsection{Abel: division of the lemniscate}
In this case we are interested in the algebraicity of the real numbers $s_0=0, s_1, \ldots, s_{l-1}$ such that
\begin{equation*}
\int\limits_0^{s_i} \frac{1}{\sqrt{1-s^{4}}} \mathrm{d}s= i/l \int\limits_0^{1} \frac{1}{\sqrt{1-s^{4}}}\mathrm{d}s.
\end{equation*}
By example \ref{alb2} and Proposition \ref{Inperiod} we see that $alb_{(0,1)}((s_i, \sqrt{1-s_i^4})) \in Alb(H_2)[4l]$, Since the Albanese map is an isomorphism, we deduce that $s_i$ is algebraic. Better, since $H_2$ is the elliptic curve with $CM$ by $\integ[i]$, we see that $s_i$ belongs to the ray class field of $\rat(i)$ with modulus $4l$, whose Galois group over $\rat(i)$ is isomorphic to $(\integ[i]/4l)^\times$. One deduces from this an analogue of Gauss' theorem regarding the divisibility of Bernoulli's lemniscate with ruler and compass stated in the introduction.

\subsection{Division of the Kiepert curve}
We are now interested in the algebraicity of the real numbers $s_0=0, s_1, \ldots, s_{l-1}$ such that
\begin{equation}\label{intddivide}
\int\limits_0^{s_i} \frac{1}{\sqrt{1-s^{6}}} \mathrm{d}s= i/l \int\limits_0^{1} \frac{1}{\sqrt{1-s^{6}}}\mathrm{d}s.
\end{equation}
In this case the hyperelliptic curve $H_3$ has genus 2, and by Proposition \ref{structalb} its Albanese variety is isogenous to $Alb(H_3/\iota\tau)^2$, where $Alb(H_3/\iota\tau)$ is isomorphic to the elliptic curve $Alb(H_3/\tau)$ with $CM$ by $\integ[\zeta_3]$. We have a commutative square
\begin{center}
\begin{tikzcd}
H_3 \arrow[d, "q"] \arrow[r, "alb_{(0, 1)}"] & Alb(X)\arrow[d, "alb(q)"] \\
H_3/\iota \tau \arrow[r] & Alb(H_3/\iota \tau)
\end{tikzcd}
\end{center}
where $q$ is the quotient map and the bottom arrow is an isomorphism because $H_3/\iota \tau$ has genus one. The differential form $\omega_0=\mathrm{d}x/y\in H^0(H_3, \Omega)$ is $\iota \tau$-invariant, therefore there is a unique differential form $\tilde{\omega} \in H^0(H_3/\iota\tau, \Omega)$ such that $q^*(\tilde{\omega})=\omega$. The map $t \mapsto \int\limits_0^{t} \frac{1}{\sqrt{1-s^{6}}} \mathrm{d}s$ gets identified with the restriction to $\{(t, \sqrt{1-t^6}), t \in [0, 1]\}$ of the map
\begin{align*}
alb(q)\circ alb_{(0,1)}: H_3 \rightarrow & Alb(H_3/\iota \tau)\\
P \mapsto q(P) \mapsto & \int\limits_{q((0,1))}^{q(P)}\tilde{\omega}.
\end{align*}
By Proposition \ref{Inperiod} we have $4(alb(q)\circ alb_{(0,1)}((1,0)))=0$. In fact, since the path $\gamma_0 - \gamma_3$ is mapped to a cycle in $H_1(H_3/\iota \tau, \integ)$ via $q$, the equality $2(alb(q)\circ alb_{(0,1)}((1,0)))$ holds. Therefore, if $s_i$ satisfies equation \eqref{intddivide}, then $alb(q)\circ alb_{(0,1)}(s_i, \sqrt{1-s_i^6}) \in Alb(H_3/\iota\tau)[2l]$. Since $q$ has degree 2 and the bottom arrow is an isomorphism we deduce that $s_i$ is algebraic, belonging to a degree 2 extension of the field generated by the $2l$-torsion points of the elliptic curve with $CM$ by $\integ[\zeta_3]$. We have proved
\begin{teo}\label{teokeipdiv}
Let $l \geq 1$, $P_0=0, P_1=r(P_1)e^{i\theta(P_1)}, \ldots, P_{6_l-1}=r(P_{6l-1})e^{i\theta(P_{6l-1})}$ be the $6l$-division points of $(C_3, 0)$. Then the numbers $r(P_i)/2^{1/3}$ are algebraic, belonging to an extension of degree at most 2 of the ray class field of $\rat(\zeta_3)$ with modulus $2l \integ[\zeta_3]$. In particular, the points $P_i$ belong to $C_3(\bar{\rat})$.
\end{teo}
\begin{rem}
In this case one does \emph{not} obtain an analogue of Gauss' theorem on the division of the circle with ruler and compass. Besides the problem that $r(P_i)=2^{1/3}s_i$ and $2^{1/3}$ is not constructible, which can be circumvented replacing $C_3$ by $2^{-1/3}C_3$, one sees that $(\integ[\zeta_3]/p)^\times$ is never a $2$-group.
\end{rem}

\begin{rem}
It is natural to wonder at this point whether the division point of $(C_n, 0)$, for $n > 3$, are algebraic. The argument given above relies on the fact that, for $n\leq 3$, the differential giving the length of arcs of $C_n$ descends to a differential on an isogeny factor of $Alb(H_n)$ of dimension 1. For $n=4$ the relevant Albanese variety splits up to isogeny as $E_1\times E_2^2$, where $E_1$ (resp. $E_2$) is the elliptic curve with $CM$ by $\integ[i]$ (resp. $\integ[\sqrt{-2}])$ (see \cite[Section 4.1.7]{paul06}). However the length differential does not descend to these elliptic curves.

For $n=p$ prime larger than 3, Proposition \ref{structalb} shows that $Alb(H_p)$ has no isogeny factor of dimension one. In view of Proposition \ref{transperab}, the length of $C_n$ \emph{cannot} be expressed as an elliptic integral (no matter which change of variables one tries to make), hence our arguments do not apply.
\end{rem}

\section{Algebraicity of division points: further examples}\label{sinsp}

\subsection{Sinusoidal spirals} The method described above adapts to show that other (pointed) algebraic curves  $C \in \mathcal{S}$ have algebraic division points. Let us consider for example the family of \emph{sinusoidal spirals}, i.e. plane curves with polar equation
\begin{equation*}
C_q: r^q=2cos(q\theta).
\end{equation*}
For every $q \in \rat$ those are algebraic curves over $\bar{\rat}$. Since $C_{-q}$ is the inverse of $C_{q}$ (with respect to $S^1$) it is enough to show this for $q=\frac{a}{b}$ with $a, b$ positive, coprime integers. Then $C_{\frac{a}{b}}$ belongs to $\mathcal{S}$, since it is the preimage of $C_\frac{1}{b}$ via the map $z \mapsto z^a$ and $C_{\frac{1}{b}}$ is the image of $C_1$ via $z \mapsto z^b$.

For $q=\frac{a}{b}$ as above the curve $C_q$ is composed of a $a$ ``leaves'' obtained by rotations with angles $\frac{2k\pi}{q}, 0 \leq k \leq a-1$, from the leaf $L$ whose points $P=(r(P), \theta(P))$ satisfy $-\frac{\pi}{2q}\leq \theta(P) \leq \frac{\pi}{2q}$. The latter is symmetric with respect to the $x$-axis and our calculation \eqref{lengtherdos} adapts in this setting to show that the length of an arc from $0$ to a point $P \in L$ with positive coordinates equals
\begin{equation*}
l(P)=\int\limits_0^{r(P)} \frac{1}{\sqrt{1-\left(\frac{r}{2^{1/q}}\right)^{2q}}} \mathrm{d}r= 2^{1/q} \int\limits_0^{\frac{r(P)}{2^{1/q}}} \frac{1}{\sqrt{1-s^{2q}}} \mathrm{d}s.
\end{equation*}

For $q=3/2$ one obtains the integral $\int\limits_0^{\frac{r(P)}{2^{1/q}}} \frac{1}{\sqrt{1-s^{3}}}\mathrm{d}s$. Since the (compactification of the) curve $y^2=1-x^3$ has genus one, a variation of the argument in the previous section applies to show that division points on $(C_{3/2}, 0)$ are algebraic (and are again related to torsion points on the elliptic curve with $CM$ by $\integ[\zeta_3]$).

\begin{figure}[h!]
\begin{center}
\includegraphics[width=0.2\textwidth]{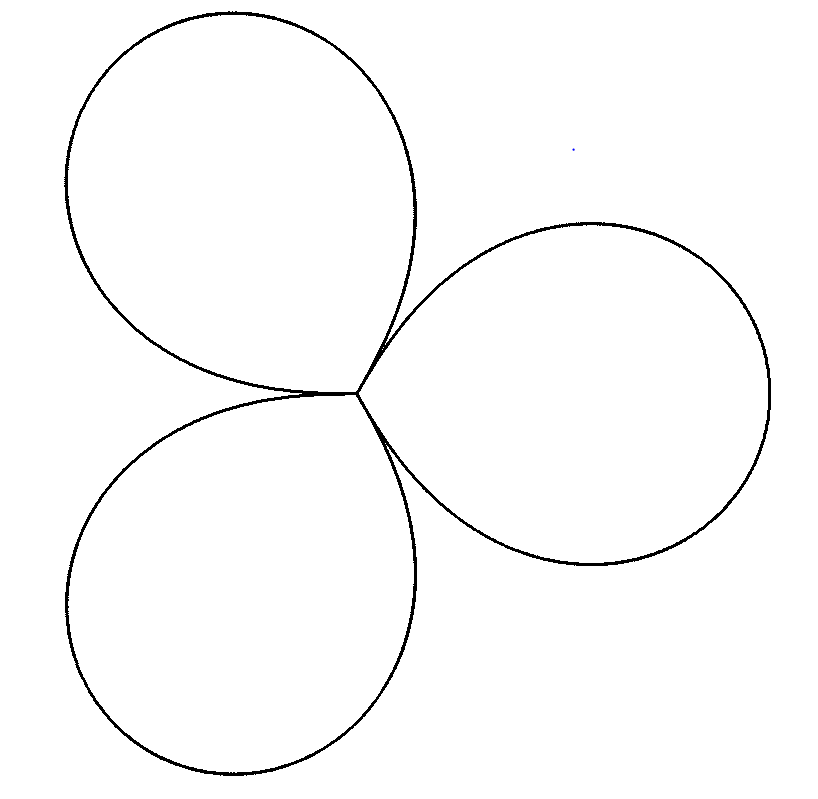}
\caption{The sinusoidal spiral $C_{\frac{3}{2}}$}
\end{center}
\end{figure}

\subsection{} In general, the argument we gave shows that division points on a pointed algebraic curve $(C, P)$ over $\bar{\rat}$ are algebraic whenever the integral giving the length of arcs of $C$ (starting from $P$) can be reduced to an elliptic integral (of the first kind). As we mentioned in the introduction, Serret discovered in 1845 \cite{ser45} an infinite family of curves having this property, of which the lemniscate is an example. He called them \emph{elliptic curves of the first kind}, and further studied them in \cite{ser45a}. One year later \cite{ser46} he found a remarkable geometric construction of these curves. This construction is also explained in \cite[Section 3.3, 4.5]{praso97}.

\subsection{Cardioid} Finally let us record the easiest example (as far as we know) of a curve with algebraic division points. It is the \emph{cardioid}, i.e. the sinusoidal spiral with parameter $q=\frac{1}{2}$. In this case one obtains, for $P$ in the upper half plane:
\begin{equation*}
l(P)=4\int\limits_0^{\frac{r(P)}{4}} \frac{1}{\sqrt{1-s}} \mathrm{d}s=-8\left(\sqrt{1-\frac{r(P)}{4}}-1\right)
\end{equation*}
and the total length is $l(C_{1/2})=16$. Algebraicity of division points follows immediately.

Let us also remark that the cardioid is an example of algebraic curve over $\bar{\rat}$ with \emph{algebraic} length.

\subsection{Length of sinusoidal spirals}\label{lensin} In fact, we can determine which of the sinusoidal spirals $C_q, \: q=\frac{a}{b} >0$ have transcendental length. Indeed, by the above discussion the total length of $C_q$ equals

\begin{equation*}
l(C_q)= 2a 2^{\frac{1}{q}} \int\limits_0^{1} \frac{1}{\sqrt{1-s^{2q}}} \mathrm{d}s=b2^{\frac{1}{q}}B\left(\frac{1}{2}, \frac{1}{2q}\right).
\end{equation*}

Therefore
\begin{enumerate}
\item If $q=\frac{1}{2k+1}, k \in \nat$ then $B(\frac{1}{2}, \frac{2k+1}{2})=\frac{\Gamma(1/2)\Gamma(1/2+k)}{\Gamma(k+1)}=\frac{\pi \cdot (1/2)(1/2+1)\cdots (k-1/2)}{k!}$. In particular $l(C_q)$ is transcendental.
\item If $q=\frac{1}{2k}, k \in \nat$, then $B(\frac{1}{2}, k)=\frac{(k-1)!}{(1/2)(1/2+1)\cdots (k-1/2)}$. In particular in this case $l(C_q)$ is algebraic.
\item For every other value of $q$ we have that $\frac{1}{2}$, $\frac{1}{2q}$ and $\frac{q+1}{2q}$ are not integers, hence $l(C_q)$ is transcendental because of the already mentioned result of Schneider.
\end{enumerate}

\section{Division points of Cassini ovals}\label{casov}

\subsection{Cassini ovals} Let us now turn our attention to Cassini ovals, i.e. curves of the form $C_P=\{z \in \comp : |P(z)|=1\}$ where $P \in \comp[T]$ is monic of degree 2. Up to a linear change of coordinates we can assume that $P$ is of the form $P(T)=T^2-a^2$ where $a$ is a positive real number. We will denote the corresponding curve by $C_a$ and assume in this section that $0<a<1$. In this case $C_a$ is connected (in the analytic topology).

\begin{figure}[h!]\label{cassov}
\begin{center}
\includegraphics[width=0.2\textwidth]{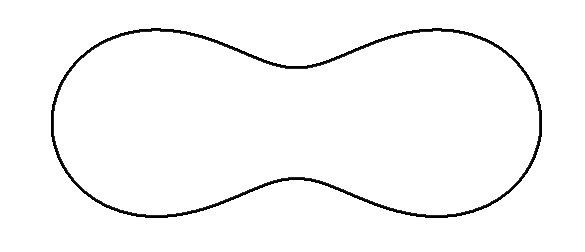}
\caption{A Cassini oval}
\end{center}
\end{figure}

Our task is to prove the following result.

\begin{teo}\label{divcasov}
Let $0<a<1$ be an \emph{algebraic} number and let $n \geq 1$. There are two points $P_n, P'_n \in C_a(\bar{\rat})$ such that the length of the shortest arc joining them equals $\frac{1}{4n}l(C_a)$.  
\end{teo}

\subsection{} The idea is again to show that the length of suitable arcs of Cassini ovals can be expressed via elliptic integrals. Our calculation is essentially the same as the one sketched in \cite{ega84}, where the trick to reduce the integral to an elliptic one is however not made explicit. Since this is crucial for us, we outline the computation in the next subsection. Before this, we should point out that the relation between arc lengths of Cassini ovals and elliptic integrals was first noticed in \cite{ser43}.

\subsection{Computation of the arc length} The equation of $C_a$ in polar coordinates is:

\begin{equation}\label{poleqcas}
r^4-2a^2r^2cos(2\theta)+a^4-1=0 \Leftrightarrow r^2= a^2cos(2\theta)+\sqrt{1-a^4sin^2(2\theta)}.
\end{equation}

For $0<\theta<\pi/2$ let $l(\theta)$ be the length of the arc joining the points on $C_a$ with angles 0 and $\theta$. Then

\begin{equation*}
l(\theta)=\int\limits_0^{\theta}r(\theta)\sqrt{1+\left(\frac{1}{2}\frac{(r(\theta)^2)'}{r(\theta)^2} \right)^2}\mathrm{d}\theta. 
\end{equation*}

Now one computes using \eqref{poleqcas}:

\begin{equation*}
\frac{1}{2}\frac{(r(\theta)^2)'}{r(\theta)^2}=-\frac{a^2 sin(2 \theta)}{\sqrt{1-a^4sin^2(2\theta)}},
\end{equation*}

hence

\begin{equation*}
l(\theta)=\int\limits_0^{\theta}r(\theta)\sqrt{\frac{1}{1-a^4 sin^2(2 \theta)}}\mathrm{d}\theta=\int\limits_0^{\theta}\frac{\sqrt{a^2cos(2 \theta)+\sqrt{a^4 cos^2(2\theta)+(1-a^4)}}}{\sqrt{a^4cos^2(2\theta)+(1-a^4)}}\mathrm{d}\theta.
\end{equation*}

Via the change of variables $u=2\theta$, and letting $b=\frac{1-a^4}{a^4}$, we obtain:
\begin{equation*}
l(\theta)=\frac{1}{2a}\int\limits_0^{u/2}\frac{\sqrt{cos(t)+\sqrt{cos^2(t)+b}}}{\sqrt{cos^2(t)+b}}\mathrm{d}t.
\end{equation*}

\subsection{} Now, for $0 \leq u \leq \pi/2$, we will compute the integral
\begin{equation}\label{sumint}
I(u)=\int\limits_0^{u}\frac{\sqrt{cos(t)+\sqrt{cos^2(t)+b}}}{\sqrt{cos^2(t)+b}}\mathrm{d}t+\int\limits_{\pi-u}^{\pi}\frac{\sqrt{cos(t)+\sqrt{cos^2(t)+b}}}{\sqrt{cos^2(t)+b}}\mathrm{d}t.
\end{equation}
In view of the above change of variables, this will give the sum of the lengths of the arcs whose end points have angles $(0, u/2)$ and $(\pi/2-u/2, \pi/2)$. If we can show that, for $u$ such that $I(u)=\frac{n-1}{n}I(\pi/2)$, $cos(u)$ is an algebraic number, it will follow that the points on $C_a$ with angles $u/2$ and $\pi/2-u/2$ are algebraic, and the arc between them has length $\frac{1}{4n}l(C_a)$.

Let us first treat the first summand $I_1(u)$ in \eqref{sumint}. Since $0\leq u \leq \pi/2$ we can make the change of variables $cos(t)=\sqrt{b}\sqrt{\frac{1}{v^2}-1}$. Then we find
\begin{equation*}
I_1(u)=a^2\int\limits_{\sqrt{\frac{b}{b+1}}}^{v(u)}\frac{\sqrt{\sqrt{b}\left(\sqrt{\frac{1}{v^2}-1}+\frac{1}{v}\right)}}{\sqrt{1-v^2}\sqrt{v^2-(1-a^4)}}\mathrm{d}v,
\end{equation*}
where $v(u)=\left(\frac{cos^2(u)}{b}+1 \right)^{-1/2}$.

On the other hand, for $\pi/2 \leq u \leq \pi$ we have $cos(u)\leq 0$, hence in the second summand $I_2(u)$ of \eqref{sumint} we can make the cange of variables $cos(t)=-\sqrt{b}\sqrt{\frac{1}{v^2}-1}$ and we find
\begin{equation*}
I_2(u)=-a^2\int\limits_{v(u)}^{\sqrt{\frac{b}{b+1}}}\frac{\sqrt{\sqrt{b}\left(-\sqrt{\frac{1}{v^2}-1}+\frac{1}{v}\right)}}{\sqrt{1-v^2}\sqrt{v^2-(1-a^4)}}\mathrm{d}v=a^2\int\limits_{\sqrt{\frac{b}{b+1}}}^{v(u)}\frac{\sqrt{\sqrt{b}\left(-\sqrt{\frac{1}{v^2}-1}+\frac{1}{v}\right)}}{\sqrt{1-v^2}\sqrt{v^2-(1-a^4)}}\mathrm{d}v.
\end{equation*}

Therefore

\begin{align*}
I(u)= & a^2b^{1/4}\int\limits_{\sqrt{\frac{b}{b+1}}}^{v(u)}\frac{\sqrt{\left(\sqrt{\frac{1}{v^2}-1}+\frac{1}{v}\right)}+\sqrt{\left(-\sqrt{\frac{1}{v^2}-1}+\frac{1}{v}\right)}}{\sqrt{1-v^2}\sqrt{v^2-(1-a^4)}}\mathrm{d}v\\
= & a^2(4b)^{1/4}\int\limits_{\sqrt{\frac{b}{b+1}}}^{v(u)}\frac{\sqrt{\frac{v+1}{v}}}{\sqrt{1-v^2}\sqrt{v^2-(1-a^4)}}\mathrm{d}v\\
= & a^2(4b)^{1/4}\int\limits_{\sqrt{1-a^4}}^{v(u)} \frac{1}{\sqrt{v(1-v)(v^2-(1-a^4))}}\mathrm{d}v.
\end{align*}

Hence
\begin{equation*}
I(\pi/2)=a^2(4b)^{1/4}\int\limits_{\sqrt{1-a^4}}^{1} \frac{1}{\sqrt{v(1-v)(v^2-(1-a^4))}}\mathrm{d}v
\end{equation*}
and
\begin{align*}
2\int\limits_{\sqrt{1-a^4}}^{1} \frac{1}{\sqrt{v(1-v)(v^2-(1-a^4))}}\mathrm{d}v= & \int\limits_{\sqrt{1-a^4}}^{1} \frac{1}{\sqrt{v(1-v)(v^2-(1-a^4))}}\mathrm{d}v+\\
&\int\limits_{1}^{\sqrt{1-a^4}} \frac{1}{-\sqrt{v(1-v)(v^2-(1-a^4))}}\mathrm{d}v.
\end{align*}

\subsection{Algebraicity of division points} Letting $E_a$ be the (hyper)elliptic curve with affine equation $y^2=x(1-x)(x^2-(1-a^4))$, we can write
\begin{equation*}
\frac{2I(\pi/2)}{a^2(4b)^{1/4}}=\int_\gamma \frac{\mathrm{d}x}{y}
\end{equation*}
where $\gamma \in H^1(E_a, \integ)$ is the sum of the paths $\gamma_1$ and $-\gamma_2$, with
\begin{align*}
\gamma_i: [\sqrt{1-a^4}, 1] & \rightarrow E_a\\
t & \mapsto \left(t, (-1)^{i-1}\sqrt{t(1-t)(t^2-(1-a^4))}\right).
\end{align*}

Since $E_a$ has genus 1, it follows that, if $u$ satisfies $I(u)=\frac{n-1}{n}I(\pi/2)$ then
\begin{equation*}
\left(v(u), \sqrt{v(u)(1-v(u))(v(u)^2-(1-a^4))}\right) \in E_a[2n].
\end{equation*}
In particular $v(u)$ is an algebraic number. It follows that $cos(u)=\sqrt{b}\sqrt{v(u)^{-2}-1}$ is also algebraic, as we had to show. This concludes the proof of Theorem \ref{divcasov}.

\subsection{Total length of Cassini ovals} For later use, let us explicitly compute the length of $C_a$, for $0<a<1$, in terms of elliptic integrals (a similar attempt has been made in \cite[(11)]{ega84}, but the result seems incorrect and in any case it is not optimal). We have
\begin{equation}\label{lencas}
l(C_a)=4\frac{1}{2a}I(\pi/2)=2\sqrt{2}(1-a^4)^{1/4}\int\limits_{\sqrt{1-a^4}}^{1} \frac{1}{\sqrt{v(1-v)(v^2-(1-a^4))}}\mathrm{d}v.
\end{equation}
Recall that the elliptic integral (of the first kind) with modulus $m$ is defined as (beware our slightly unusual convention here):
\begin{equation}\label{ellint}
K(m)=\int\limits_{0}^{1} \frac{\mathrm{d}t}{\sqrt{(1-t^2)(1-mt^2)}}=\int\limits_{0}^{1} \frac{\mathrm{d}v}{2\sqrt{v(1-v)(1-mv)}}.
\end{equation}

On the other hand, via the change of variable $1/v=u$ the integral in \eqref{lencas} becomes
\begin{equation*}
\int\limits_{1}^{\frac{1}{\sqrt{1-a^4}}}\frac{\mathrm{d}u}{\sqrt{(u-1)(1-u\sqrt{1-a^4})(1+u\sqrt{1-a^4})}}.
\end{equation*}
Via the further change of variables $w=\frac{u-1}{\frac{1}{\sqrt{1-a^4}}-1}$ the above integral becomes
\begin{equation*}
\frac{1}{(1-a^4)^{1/4}(\sqrt{1+\sqrt{1-a^4}})}\int\limits_{0}^{1}\frac{\mathrm{d}w}{\sqrt{w(1-w)(1+\frac{1-\sqrt{1-a^4}}{1+\sqrt{1-a^4}}w)}},
\end{equation*}
hence
\begin{equation*}
l(C_a)=\frac{4\sqrt{2}}{\sqrt{1+\sqrt{1-a^4}}}K\left(\frac{1-\sqrt{1-a^4}}{-1-\sqrt{1-a^4}}\right).
\end{equation*}
Setting $s=\frac{1-\sqrt{1-a^4}}{2}$ we get
\begin{equation*}
l(C_a)=\frac{4}{\sqrt{1-s}}K\left(\frac{s}{s-1}\right).
\end{equation*}
Because of the identity
\begin{equation}\label{ellpfaff}
K(s)=\frac{1}{\sqrt{1-s}}K\left(\frac{s}{s-1}\right)
\end{equation}
we finally find
\begin{equation}\label{lencasell}
l(C_a)=4K\left(\frac{1-\sqrt{1-a^4}}{2}\right).
\end{equation}

It is a classical result of Schneider \cite[Satz 15(III. Fassung)]{sch57} that $K(m)$ is transcendental whenever $m$ is algebraic. Together with the above formula this implies the following
\begin{prop}\label{tracas}
If $a \in \bar{\rat}$ then $l(C_a)$ is transcendental.
\end{prop}
\begin{proof}
If $0<a<1$ this follows from equation \ref{lencasell} and Schneider's theorem. This implies the result for every $a$ since one has, for $0<a<1$, $l(C_{1/a})=a\cdot l(C_a)$ (the reader can check this adapting the previous calculation or find a proof of a more general statement in the next section).
\end{proof}

\section{Remarks on the length of regular lemniscates}

\begin{defin}
Let $a>0$ be a real number, and $k\geq 1$ an integer. We call the curve
\begin{equation*}
C_{a, k}=\{z \in \comp: |z^k-a^k|=1\}
\end{equation*}
a \emph{regular polynomial lemniscate}.
\end{defin}

\subsection{} The name comes from the fact that $C_{a, k}$ is the level set of the polynomial $T^k-a^k$, whose roots are the vertices of a regular polygon in the plane. For $a \in \bar{\rat}$ the curve $C_{a, k}$ is in $\mathcal{S}$. For example for $a=1$ we obtain Erd\"{o}s lemniscates. If $k=2$ we recover Cassini ovals. These curves were first studied in \cite{ser43}, \cite{ser43a}.

The equation of $C_{a, k}$ in polar coordinates is

\begin{equation*}
r^{2k}-2a^kr^kcos(k\theta)+a^{2k}-1=0.
\end{equation*}

The aim of this section is to compute the length of $C_{a, k}$, expressing it in terms of hypergeometric functions. As a consequence of our computation we will be able to:
\begin{enumerate}
\item verify that, for given $k$, the longest curve among the $C_{a, k}$ is the Erd\H{o}s lemniscate, in accordance with Erd\H{o}s conjecture stating that $C_{1, k}$ is the longest curve of the form $P^{-1}(S^1)$, for $P \in \comp[T]$ monic of degree $k$.
\item Compare our result with the one in the previous section for $k=2$ and rediscover a very classical formula expressing the hypergeometric function $_2F_1\left(\frac{1}{4}, \frac{3}{4}, 1, x \right)$ in terms of elliptic integrals, which explains a posteriori why Ramanujan's theory of elliptic functions to alternative bases works in signature 4 (see \cite[Chapter 33]{ber97}).
\item Check that the length of $C_{a, k}$ is transcendental when $a \in \bar{\rat}$, and construct, using Proposition \ref{transperab} and $(2)$, a family of genus 2 curves over $\bar{\rat}$ whose Jacobian is isogenous to a product of elliptic curves. One says that such a curve has \emph{totally split Jacobian}.
\end{enumerate}

\subsection{} The behaviour of the length of the curves $C_{a, k}$ (or of closely related curves) for $a$ varying has been studied qualitatively in \cite{ega84} and \cite{wa06}. This length was also calculated in \cite{bu91}, and a more general formula has been established in \cite{kuta00}. It seems however that it went unnoticed that the expression of $l(C_{a, k})$ in terms of hypergeometric functions follows easily from the formula given by Serret in \cite{ser43a} (of which \eqref{serform} is a special case). We deduce this formula from a straightforward generalization of an old computation by Matz \cite{ma95} on the length of Cassini ovals. We outline this computation here, generalised to our situation.

\subsection{} From the polar equation of $C_{a, k}$ one deduces that
\begin{align*}
cos(k\theta)=&\frac{a^{2k}-1}{2a^k}\frac{1}{r^k}+\frac{r^k}{2a^k}\Rightarrow r\frac{\mathrm{d}\theta}{\mathrm{d}r}=\frac{1}{2sin(k\theta)}\left( \frac{a^{2k}-1}{a^kr^k}-\frac{r^k}{a^k} \right)\\
\Rightarrow 1+\left(r\frac{\mathrm{d}\theta}{\mathrm{d}r}\right)^2=&\frac{(2r^k)^2}{(r^{2k}-(a^k-1)^2)((1+a^k)^2-r^{2k})}.
\end{align*}

\subsection{} Let us first suppose $0<a<1$. In this case $C_{a, k}$ is obtained from the arc between the points with angles $-\frac{\pi}{k}$ and $\frac{\pi}{k}$ by rotations of $\frac{2j\pi}{k}$, $j=0, \ldots, k-1$. This arc in turn is symmetric with respect to the axis $y=0$.

\begin{figure}[h!]
\begin{center}
\includegraphics[width=0.2\textwidth]{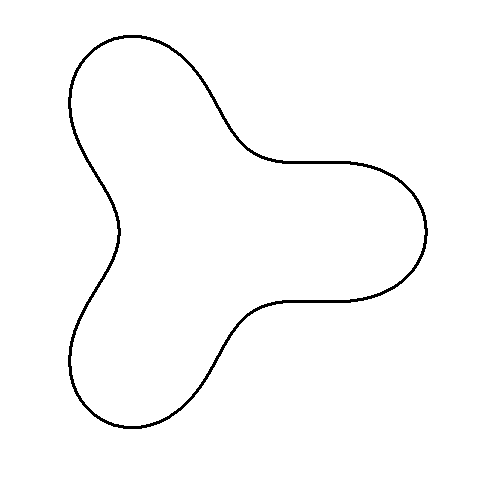}
\caption{$C_{a, 3}, 0<a<1$}
\end{center}
\end{figure}

It follows that

\begin{equation}\label{serform}
l(C_{a,k})=2k\int\limits_{(1-a^k)^{1/k}}^{(1+a^k)^{1/k}}\frac{2r^k\mathrm{d}r}{\sqrt{(r^{2k}-(a^k-1)^2)((1+a^k)^2-r^{2k})}}.
\end{equation}

Set $r^{2k}=(1-a^k)^2sin^2(\phi)+(1+a^k)^2cos^2(\phi)$. A straightforward computation yields:
\begin{equation}\label{matzfor}
l(C_{a,k})=\frac{4}{(1+a^k)^{\frac{k-1}{k}}}\int\limits_{0}^{\pi/2}\frac{\mathrm{d}\phi}{\left( 1-\frac{4a^k}{(1+a^k)^2}sin^2(\phi) \right)^{\frac{k-1}{2k}}};
\end{equation}
for $k=2$ this recovers indeed (a special case of) the final formula in \cite{ma95}.

\subsection{} In the case $a>1$ the curve $C_{a,k}$ has $k$ isometric connected components. The one crossing the positive $x$-axis is symmetric with respect to it.

\begin{figure}[h!]
\begin{center}
\includegraphics[width=0.2\textwidth]{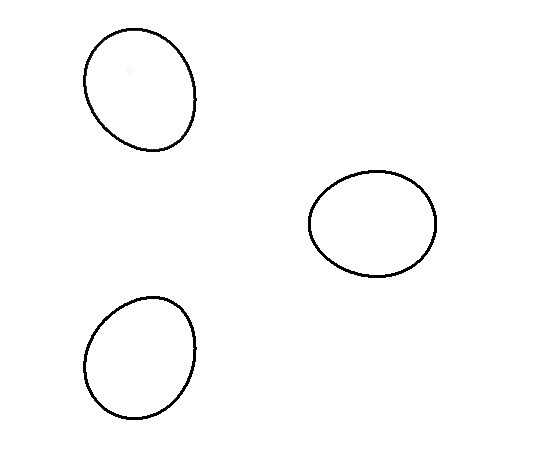}
\caption{$C_{a, 3}, a>1$}
\end{center}
\end{figure}

Hence we find

\begin{equation*}
l(C_{a,k})=2k\int\limits_{(a^k-1)^{1/k}}^{(1+a^k)^{1/k}}\frac{2r^k\mathrm{d}r}{\sqrt{(r^{2k}-(a^k-1)^2)((1+a^k)^2-r^{2k})}}.
\end{equation*}

The same computation as before works; setting $b=a^{-1}$ one finds

\begin{equation}\label{lenggr1}
l(C_{a,k})=b^{k-1}\frac{4}{(1+b^k)^{\frac{k-1}{k}}}\int\limits_{0}^{\pi/2}\frac{\mathrm{d}\phi}{\left( 1-\frac{4b^k}{(1+b^k)^2}sin^2(\phi) \right)^{\frac{k-1}{2k}}}=\frac{1}{a^{k-1}}l\left(C_{\frac{1}{a}, k}\right).
\end{equation}

\subsection{Gauss' hypergeometric functions} It is now easy to express the integral in \eqref{matzfor} in terms of hypergeometric functions. Let us recall the key properties of these functions which we will need.

For $p, q \in \real$ and $r \in \real \setminus \integ_{\leq 0}$ the hypergeometric function $_2F_1(p, q, r; z)$ is defined on the unit disc $|z|<1$ by the power series
\begin{equation*}
_2F_1(p, q, r; z)=\sum_{n=0}^\infty \frac{(p)_n(q)_n}{(r)_n}\frac{z^n}{n!}
\end{equation*}
where, for $x \in \real$:
\begin{equation*}
(x)_n=\begin{cases}
1 \text{ if } n=0\\
x(x+1)\cdots(x+n-1) \text{ if } n>0.
\end{cases}
\end{equation*}

It enjoys the following properties (among countless others):
\begin{description}
\item[Symmetry] $_2F_1(p, q, r; z)={_2F_1}(q, p, r; z)$.
\item[Euler integral representation] $B(q, r-q){_2F_1}(p, q, r; z)=\int\limits_{0}^{1}t^{q-1}(1-t)^{r-q-1}(1-zt)^{-p}\mathrm{d}t$. In particular we have $_2F_1(\frac{1}{2}, \frac{1}{2}, 1; m)=\frac{2}{\pi}K(m)$ where $K(m)$ is the elliptic integral as defined in \eqref{ellint}.
\item[Pfaff transformation] $_2F_1(p, q, r; z)=(1-z)^{-p} {_2F_1}(p, r-q, r; \frac{z}{z-1})$. In particular, for $p=q=\frac{1}{2}$ and $r=1$ we find the transformation law \eqref{ellpfaff}.
\item[Quadratic transformation] $_2F_1(p, q, 2q; \frac{4z}{(1+z)^2})=(1+z)^{2p}{_2F_1}(p, p-q+\frac{1}{2}, q+\frac{1}{2}; z^2)$ (see \cite[(3.1.11), pag. 128]{anasro00}).
\end{description}

\subsection{Length of $C_{a, k}$}\label{lghcak} Let us come back to the length of $C_{a, k}$ for $0<a<1$ as computed in \eqref{matzfor}. The change of variable $sin^2(\phi)=t$ yields, using the integral representation and the quadratic transformation:
\begin{align*}
l(C_{a,k})=&\frac{2}{(1+a^k)^{\frac{k-1}{k}}}\int\limits_{0}^{1}\frac{\mathrm{d}t}{\left( 1-\frac{4a^k}{(1+a^k)^2}t \right)^{\frac{k-1}{2k}}\sqrt{t}\sqrt{1-t}}\\
=&\frac{2\pi}{(1+a^k)^{\frac{k-1}{k}}}{_2F_1}\left(\frac{k-1}{2k}, \frac{1}{2}, 1; \frac{4a^k}{(1+a^k)^2}\right)=2\pi {_2F_1}\left(\frac{k-1}{2k}, \frac{k-1}{2k}, 1; a^{2k} \right).
\end{align*}

\begin{rem}
Recall that for $a=1$ the curve $C_{a, k}$ is the Erd\H{o}s lemniscate with $k$ leaves, whose length was computed in \eqref{totlengtherdos}. This suggests that the following equality should hold true:

\begin{equation*}
2\pi {_2F_1}\left(\frac{k-1}{2k}, \frac{k-1}{2k}, 1; 1 \right)=2^{\frac{1}{k}}B\left(\frac{1}{2}, \frac{1}{2k}\right).
\end{equation*}

This is indeed the case: the above formula follows from Legendre duplication formula for the Gamma function and from the following identity attributed to Gauss, which is a consequence of Euler integral representation:
\begin{equation*}
{_2F_1}(p, q, r; 1)=\frac{\Gamma(r)\Gamma(r-p-q)}{\Gamma(r-p)\Gamma(r-q)}.
\end{equation*}
\end{rem}

\subsection{Case k=2} In this case, comparing the formula we just obtained with \eqref{lencasell}, we find the equality:
\begin{equation*}
{_2F_1}\left( \frac{1}{4}, \frac{1}{4}, 1; a^4\right)=\frac{2}{\pi}K\left(\frac{1-\sqrt{1-a^4}}{2}\right).
\end{equation*}
Using the Pfaff transformation and setting $a^4=z$ gives:
\begin{equation}\label{hypgeoell}
{_2F_1}\left( \frac{1}{4}, \frac{3}{4}, 1; \frac{z}{z-1}\right)=\frac{2(1-z)^{\frac{1}{4}}}{\pi}K\left(\frac{1-\sqrt{1-z}}{2}\right)
\end{equation}
This the promised (very classical) relation between elliptic integrals and the hypergeometric function ${_2F_1}\left( \frac{1}{4}, \frac{3}{4}, 1; -\right)$.

\subsection{Behaviour of $l(C_{a,k})$} One deduces immediately from the formula obtained in \ref{lghcak} and from equation \eqref{lenggr1} the following information on the behaviour of the length of $C_{a,k}$:

\begin{prop}\label{erdconjspec}
Fix an integer $k \geq 2$. The function
\begin{align*}
\real_{\geq 0} \rightarrow & \real_{\geq 0}\\
a \mapsto & l(C_{a, k})
\end{align*}
attains its maximum at $a=1$. Moreover its restriction to the interval $[0, 1[$ is strictly increasing and convex.
\end{prop}

\subsection{Transcendence of lengths} Let us now investigate the transcendence of $l(C_{a,k})$ for $a \in \bar{\rat}$. Up to changing coordinates we may assume that $a \in \real$, and because of equation \eqref{lenggr1} we can restrict ourselves to $0<a<1$. Letting $b=\frac{4a^k}{(1+a^k)^2}$ in \ref{lghcak}, we are interested in the transcendence of the integral

\begin{equation*}
\int\limits_{0}^{1}\frac{\mathrm{d}t}{\left( 1-bt \right)^{\frac{k-1}{2k}}\sqrt{t}\sqrt{1-t}}.
\end{equation*}

This kind of integral has been studied by Wolfart in his work on the transcendence of values of the hypergeometric function \cite{wol88}. Let $X_{b, k}$ be the smooth projective curve birational to the affine curve $y^{2k}=(1-bx)^{k-1}x^k(1-x)^k$. Then, as explained in \cite[Section 2.2]{wol88} and \cite[Section 8]{archi03}, the differential $\frac{\mathrm{d}x}{y}$ gives a holomorphic differential $\eta$ on $X_{b, k}$, and the above integral is an algebraic multiple of a non zero period of this differential. Another consequence of the already mentioned analytic subgroup theorem is a very general transcendence result for non zero periods of differential forms over $\bar{\rat}$ \cite[Theorem 6.8, Corollary 6.9]{bawu08}, widely generalizing Schneider's work on elliptic integrals and beta integrals:

\begin{teo}(W\"{u}stholz)\label{alltransc}
Let $C$ be a smooth projective curve over $\bar{\rat}$ and $\omega$ a rational differential form on $C$. Then every non zero period of $\omega$ is transcendental.
\end{teo}

Let us mention that the recent work \cite{huwu18} further generalises the previous theorem. Applying it in our case we deduce the following generalization of Propositions \ref{transcerdos} and \ref{tracas}.

\begin{prop}\label{transgen}
If $a \in \bar{\rat}$ and $k \geq 1$ then the length of $C_{a, k}$ is transcendental.
\end{prop}

\begin{rem}
We do not know which Serret curves have transcendental length. For example, is it true that all polynomial lemniscates in $\mathcal{S}$ have transcendental length? Theorem \ref{alltransc} would seem to suggest that most curves in $\mathcal{S}$ should enjoy this property; however our calculations with sinusoidal spirals in \ref{lensin} provided an infinite number of counterexamples. The issue is that a general relation between the length of Serret curves and periods of suitable curves does not appear to be transparent.
\end{rem}

Finally, our computation in the case $k=2$ joint with Proposition \ref{transperab} allows to construct a family of genus two curves with totally split Jacobian.

\begin{prop}\label{jactotspl}
Let $a \in \bar{\rat}\cap \real$ and $b=\frac{4a^2}{(1+a^2)^2}$. Then the smooth projective model of the affine curve $y^4=(1-bx)x^2(1-x)^2$ is a genus two curve whose Jacobian is isogenous to a product of elliptic curves.
\end{prop}

\begin{proof}
Let $X$ be the smooth projective curve birational to $y^4=(1-bx)x^2(1-x)^2$. By \cite[Section 8, pag. 927]{archi03} the curve $X$ has genus 2. As explained above, the differential $\frac{\mathrm{d}x}{y}$ induces a holomorphic differential $\eta \in H^0(X, \Omega)$. We proved that the length $l(C_{a, 2})$ is an algebraic multiple of a period of $\eta$. On the other hand by equation \eqref{lencasell} this length is also an algebraic multiple of a period of an elliptic curve defined over $\bar{\rat}$. It follows from Proposition \ref{transperab} that the Jacobian of $X$ has an isogeny factor isomorphic to an elliptic curve, and the proof is complete.
\end{proof}

\bibliographystyle{amsalpha}
\bibliography{bibcurve}

\Addr

\end{document}